\documentclass{article}%
\usepackage{amsfonts}
\usepackage{amssymb}%
\usepackage{amsmath}%
\usepackage{graphicx}
\providecommand{\U}[1]{\protect\rule{.1in}{.1in}}
\newtheorem{theorem}{Theorem}

\newtheorem{corollary}[theorem]{Corollary}

\newtheorem{lemma}[theorem]{Lemma}

\newenvironment{proof}[1][Proof]{\noindent\textbf{#1.} }{\ \rule{0.5em}{0.5em}}
\begin{document}

\title{\textsf{Improving Cauchy's Integral Theorem in Constructive Analysis}}
\author{\textsf{Douglas S. Bridges}}
\maketitle

\begin{abstract}%
\noindent
\textsf{In his constructive development of complex analysis, Errett Bishop
used restrictive notions of homotopy and simple connectedness. Working in
Bishop-style constructive mathematics, we prove Cauchy's integral theorem
using the standard notions of such properties. In consequence, Bishop's
theorems in Chapters 5 of \cite{Bishop,BB} hold under our more normal, less
restrictive, definitions.}

\end{abstract}%

\setcounter{secnumdepth}{0}%
\normalfont\sf

\subsection{Introduction}

The primary purpose of this note\footnote{%
\normalfont\sf
MSC 2020 Classifications: 03F60 (Primary), 30Exx (Secondary)} is to provide a
constructive (in Errett Bishop's sense\footnote{\textsf{Bishop's constructive
mathematics is, informally, mathematics carried out using intutionistic logic
and a suitable framework of type- or set-theory \cite[Chs. 1-2]{Handbook}
,\cite{BMorse}}}) proof of \emph{Cauchy's integral theorem} in the form:

\begin{theorem}
\label{oct14t1}Let the piecewise differentiable closed paths $\gamma_{0}$ and
$\gamma_{1}\ $have common parameter interval and be homotopic in the open set
$U\subset\mathbb{C}$. Then $\int_{\gamma_{0}}f=\int_{\gamma_{1}}f $ for each
analytic function $f$ on $U$.
\end{theorem}%

\noindent
The reader familiar with Bishop's development of complex analysis may feel
that this theorem has already been proved as Theorem (3.12) on page 141 of
\cite{BB}. However, Bishop worked with a more restrictive notion of homotopy,
and hence of simple connectivity, than the one that is standard in topology,
whereas our proof uses the standard notions. To clarify the distinction
between our work below and Bishop's, let's provide some definitions.

Let $\mathbb{F}$ denote either the real line $\mathbb{R}$ or the complex plane
$\mathbb{C}$. For each located subset $K$ of $\mathbb{F}$ and each $r>0 $ we
define%
\[
K_{r}\equiv\{z\in\mathbb{F}:\rho(z,K)\leq r\},
\]
where the distance from $z$ to $K,$%
\[
\rho(z,K)\equiv\inf\left\{  \left\vert z-\zeta\right\vert :\zeta\in
\mathbb{F}\right\}  \text{,}
\]
exists, by definition of located. If $K$ is totally bounded (and hence
located), then $K_{r}$ is compact for each $r>0$; if also $K$ is subset of an
open set $U\subset\mathbb{F}$ such that $K_{r}\subset U$ for some $r>0$, we
say that $K$ is \emph{well contained}\textbf{\ }in\textbf{\ }$U$, which
property we denote by $K\subset\subset U$.

We assume that the reader is familiar with, or can access, Bishop's work on
constructive complex analysis, as can be found in Sections 1--3 of Chapter 5
in \cite{Bishop} and \cite{BB}. However, in contrast to Bishop,\footnote{%
\normalfont\sf
Bishop requires that every path be piecewise differentiable, by definition.}
we shall apply the term \emph{path} to any (uniformly) continuous function
$\gamma$ from a proper compact \emph{parameter interval} of $\left[
a,b\right]  $ in the real line $\mathbb{R}$ into the complex plane
$\mathbb{C}$. The \emph{carrier }of $\gamma$, written $\mathsf{car}(\gamma)$,
is the closure of the range of $\gamma$; since the latter set is totally
bounded, $\mathsf{car}(\gamma)$ is compact and hence located. We say that
$\gamma$ \emph{lies in,}\textbf{\ }or is a\textbf{\ }\emph{path in,} the set
$S\subset\mathbb{C}$ if either $S$ is compact and $\mathsf{car}(\gamma)\subset
S$ or else $S$ is open and $\mathsf{car}(\gamma)\subset\subset S$. If
$\gamma(a)=\gamma(b)$, then $\gamma$ is called a \emph{closed path}.

Let $\gamma_{0}$ and $\gamma_{1}$ be closed paths, with common parameter
interval $\left[  a,b\right]  $, in a compact subset set $K$ of $\mathbb{C}$.
We say that these paths are \emph{homotopic in}\textbf{\ }$K$ if there exists
a continuous function $\sigma:[0,1]\times\left[  a,b\right]  \rightarrow K$
such that for each $t\in\left[  0,1\right]  $, the function $\sigma
_{t}:\left[  a,b\right]  \rightarrow\mathbb{C}$ defined by $\sigma
_{t}(x)\equiv\sigma(t,x)$ is a closed path, $\sigma_{0}=\gamma_{0}$, and
$\sigma_{1}=\gamma_{1}$. The function $\sigma$ is then called a
\emph{homotopy}\textbf{\ }of $\gamma_{0}$ and $\gamma_{1}$. If also
$U\subset\mathbb{C}$ is open and $K\subset\subset U$, then $\gamma_{0}$ and
$\gamma_{1}$ are said to be \emph{homotopic in}\textbf{\ }$U$. If $\gamma_{1}$
is a constant path, then $\gamma_{0}$ is said to be \emph{null-homotopic. }

A subset $U$ of $\mathbb{C}$ is

\begin{itemize}
\item \emph{path connected}\textbf{\ }if any two points of $U$ can be joined
by a path in $U$;

\item \emph{connected}\textbf{\ }if for any two inhabited open subsets
$A,B\ $of $U$ with $U=A\cup B$ there exists $z\in$ $A\cap B$;

\item \emph{simply connected}\textbf{\ }if it is path connected and every
closed path in $U$ is null-homotopic.
\end{itemize}%

\noindent
Note that path connected implies connected and that, \emph{classically,} if
$U$ is both open and connected, then it is path connected.\cite[page 221,
(9.7.2)]{Dieudonne}

Now we can clarify the difference between our definitions and Bishop's. First,
he deals only with the case where $U$ is an open subset of $\mathbb{C} $. In
the definition of `homotopy' he requires that all the paths $\sigma_{t}$ be
piecewise differentiable; he calls $U$ connected if any two points of $U$ can
be joined by a piecewise differentiable path in $U$; and finally, he calls $U$
simply connected if it is connected (in his sense) and every closed path in
$U$ is null-homotopic (in his sense).

\subsection{Extending Cauchy's theorem}

We turn now towards proving Cauchy's integral theorem using our more standard
notion of homotopy. To that end, we first recall that the \emph{norm} of a
uniformly continuous function on a compact set $K\subset\mathbb{C}$ is%
\[
\left\Vert f\right\Vert \equiv\sup\left\{  \left\vert f(x)\right\vert :x\in
K\right\}  ,
\]
and we have two lemmas.

\begin{lemma}
\label{sep02l1}If $f:\left[  0,1\right]  \rightarrow\mathbb{C}$ is a
continuous mapping such that $f(0)=f(1)$, then for each $\varepsilon>0$ there
exists a piecewise differentiable function $g:[0,1]\rightarrow\mathbb{C}$ such
that $g(0)=f(0)=g(1)$ and $\left\Vert f-g\right\Vert \leq\varepsilon$.
\end{lemma}

\begin{proof}
Given $\varepsilon>0$, let $\delta\in\left(  0,1\right)  $ be such that if
$x,x^{\prime}\in\left[  0,1\right]  $ and $\left\vert x-x^{\prime}\right\vert
<\delta$, then $\left\vert f(x)-f(x^{\prime})\right\vert <\varepsilon/3$.
Choose points $x_{0}=0<x_{1}<\cdots<x_{n}=1$ such that $\left\vert
x_{i+1}-x_{i}\right\vert <\delta$ for $0\leq i<n$. For such $i$ and for
$x_{i}\leq x\leq x_{i+1}$ define
\begin{align*}
\lambda_{i}(x)  & =\frac{x-x_{i}}{x_{i+1}-x_{i}},\\
g_{i}(x)  & =(1-\lambda_{i}(x))f(x_{i})+\lambda_{i}(x)f(x_{i+1}).
\end{align*}
Then $\lambda_{i}$, and hence $g_{i}$, is uniformly continuous and
differentiable on $\left[  x_{i},x_{i+1}\right]  $, $g_{i}(x_{i})=f(x_{i})$,
and $g_{i}(x_{i+1})=f(x_{i+1})$. It readily follows that there exists a
piecewise differentiable, and hence (uniformly) continuous function $g:\left[
0,1\right]  \rightarrow\mathbb{C}$ such that $g(x)=g_{i}(x)$ whenever $0\leq
i<n$ and $x\in\lbrack x_{i},x_{i+1}]$. For such $i$ and $x$ we have
$\left\vert x-x_{i}\right\vert <\delta$, so $\left\vert f(x)-f(x_{i}%
)\right\vert <\varepsilon/3$ and therefore%
\begin{align*}
\left\vert f(x)-g(x)\right\vert  & \leq\left\vert f(x)-f(x_{i})\right\vert
+\left\vert g(x)-f(x_{i})\right\vert \\
& <\varepsilon/3+\lambda_{i}(x)\left\vert f(x_{i+1})-f(x_{i})\right\vert \\
& \leq\varepsilon/3+\lambda_{i}(x)\varepsilon/3<2\varepsilon/3\text{.}%
\end{align*}
It follows from this, the density of $%
{\textstyle\bigcup\limits_{k=1}^{n-1}}
\left[  x_{k},x_{k+1}\right]  $ in $\left[  0,1\right]  $, and the uniform
continuity of $f$ and $g$ that $\left\Vert f-g\right\Vert \leq2\varepsilon
/3<\varepsilon$.
\end{proof}%

\bigskip
We state, without proof, a beautiful lemma of Bishop \cite[page 140,
(3.10)]{BB}.

\begin{lemma}
\label{oct10l1}Let $U\subset\mathbb{C}$ be open, $\varepsilon>0$, and $K$ a
compact set such that $K_{\varepsilon}\subset\subset U$. Let $\gamma_{0} $ and
$\gamma_{1}$ be closed, piecewise differentiable paths in $K$ with common
parameter interval $\left[  a,b\right]  $, such that~$\left\vert \gamma
_{0}(t)-\gamma_{1}(t)\right\vert \leq\varepsilon/2$ for all $t\in\left[
a,b\right]  $. Then $\int_{\gamma_{0}}f=\int_{\gamma_{1}}f$ for each analytic
function $f$ on $U$.
\end{lemma}%

\smallskip
We are now in a position to give our proof of Theorem \ref{oct14t1}.%

\bigskip

\begin{proof}
We may assume that $\gamma_{0}$ and $\gamma_{1}$ are defined on $\left[
0,1\right]  $. There exist a compact set $K\subset\subset U$ and a homotopy
$\sigma:\left[  0,1\right]  ^{2}\rightarrow\mathbb{C}$ between $\gamma_{0}$
and $\gamma_{1}$ in $K$. Pick $\varepsilon>0$ such that $K_{\varepsilon
}\subset\subset U$, and then $\delta\in(0,1)$ such that if $\left(
t,x\right)  ,(t^{\prime},x^{\prime})\in\left[  0,1\right]  ^{2}$ and
$\left\vert \left(  t,x\right)  -(t^{\prime},x^{\prime})\right\vert <\delta$,
then $\left\vert \sigma\left(  t,x\right)  -\sigma(t^{\prime},x^{\prime
})\right\vert <\varepsilon/6$. Choose points $t_{0}=0<t_{1}<\cdots<t_{n}=1$
such that $\left\vert t_{i+1}-t_{i}\right\vert <\delta$; then $\left\Vert
\sigma_{t_{i}}-\sigma_{t_{i+1}}\right\Vert \leq\varepsilon/6$. By Lemma
\ref{sep02l1}, for $1<i<n-1$ there exists a piecewise differentiable function
$\phi_{i}:[0,1]\rightarrow\mathbb{C}$ such that $\phi_{i}(0)=\sigma_{t_{i}%
}(0)=\sigma_{t_{i}}(1)=\phi_{i}(1)$ and $\left\Vert \phi_{i}-\sigma_{t_{i}%
}\right\Vert \leq\varepsilon/6$. Then%
\[
\left\Vert \phi_{i}-\phi_{i+1}\right\Vert \leq\left\Vert \phi_{i}%
-\sigma_{t_{i}}\right\Vert +\left\Vert \sigma_{t_{i}}-\sigma_{t_{i+1}%
}\right\Vert +\left\Vert \phi_{i+1}-\sigma_{t_{i+1}}\right\Vert <\varepsilon
/2.
\]
Also,%
\[
\left\Vert \gamma_{0}-\phi_{1}\right\Vert \leq\left\Vert \sigma_{t_{0}}%
-\sigma_{t_{1}}\right\Vert +\left\Vert \sigma_{t_{1}}-\phi_{1}\right\Vert
\leq\varepsilon/3<\varepsilon/2
\]
and%
\[
\left\Vert \phi_{n-1}-\gamma_{1}\right\Vert \leq\left\Vert \phi_{n-1}%
-\sigma_{t_{n-1}}\right\Vert +\left\Vert \sigma_{t_{n-1}}-\sigma_{t_{n}%
}\right\Vert \leq\varepsilon/3<\varepsilon/2.
\]
It follows from Lemma \ref{oct10l1} that%
\[
\int_{\gamma_{0}}f=\int_{\phi_{1}}f=\int_{\phi_{2}}f=\cdots=\int_{\phi_{n-1}%
}f=\int_{\gamma_{1}}f\text{.}
\]%
\hfill

\end{proof}

\begin{corollary}
\label{oct10c1}If $\gamma$ is a piecewise differentiable, null-homotopic,
closed path in the open set $U\subset\mathbb{C}$, then $\int_{\gamma
}f(z)\mathsf{d}z=0$ for each function $f$ analytic on $U$.
\end{corollary}

\begin{corollary}
\label{oct14c1}If $\gamma$ is a piecewise differentiable closed path in a
simply connected open set $U$ in the complex plane, then $\int_{\gamma
}f(z)\mathsf{d}z=0$ for each function $f$ analytic on $U$.
\end{corollary}

It follows from our work that, as can be verified by the reader, Bishop's
results that depend on his version of Theorem \ref{oct14t1} \cite[page 141,
(3.14)]{BB} will hold if we interpret `homotopy' and `simply connected'
according to our (standard) definitions, and Bishop's `connected' as our `path connected'.%

\bigskip

%

\bigskip
%

\noindent
\textbf{Author's address: }School of Mathematics \& Statistics, University of
Canterbury, Christchurch 8140, New Zealand \ \ \ \ \ \ \
\hfill
\textbf{email: }\texttt{dsbridges.math@gmail.com}
\end{document}